\documentclass[12pt]{amsart}
\usepackage{amsmath}
\usepackage{color}

\newtheorem{theor}{\indent\sc Theorem}[section]
\newtheorem{corol}[theor]{\indent\sc Corollary}
\newtheorem{lemma}[theor]{\indent\sc Lemma}   
\newtheorem{prop}[theor]{\indent\sc Proposition}

\newtheorem{dfn}[theor]{\indent\sc Definition}

\newtheorem{example}{\indent\sc {Example}}
\newtheorem{rema}{\indent\sc{Remark}}

\newcommand{\real}{{\mathbb R}}
\newcommand{\lpr}[2]{\langle{#1},{#2}\rangle}
\newcommand{\mink}{{\real}_{1}^{4}} 
\newcommand{\complex}{\mathbb{C}}

\newcommand{\submink}{\mathbb{L}^{3}} 

\newcommand{\euclidean}{{\mathbb{E}}^{3}}

\newcommand {\M}{{\mathcal M}}

\newcommand {\A}{{\mathcal A}}

\newcommand {\C}{{\mathcal C}}
\newcommand {\F}{{\mathcal F}}

\title{ Minimal  Spacelike Surfaces and the Graphic Equations in $\mink$}
\author{M. P. Dussan,  A. P. Franco Filho, R. S. Santos} 

\begin{document}

\begin{abstract}

In this paper we study an extension of the Bernstein Theorem for minimal spacelike surfaces of the four dimensional Minkowski 
vector space form and we obtain the class of those surfaces which are also graphics and have non-zero Gauss curvature. That is the class of entire solutions of a system of two elliptic non-linear equations that is an extension of the 
equation of minimal graphic of $\mathbb R^3$. Therefore, we prove that the so-called Bernstein property does not  hold in general for  the case 
of graphic spacelike surfaces in
$\mathbb R^4_1$.  In addition, we also obtain explicitly the conjugated minimal spacelike surface, and identify the necessary conditions to extend continuously a local solution
of the generalized Cauchy-Riemann equations.
\end{abstract}

\maketitle
{\bf Keywords}:  Minimal spacelike surface, Bernstein Theorem, Weierstrass representation 

{\bf MSC}: 53C42; 53C50

\section{Introduction}

  One relevant classic result in the context of the global geometry of spacelike surfaces it is the Bernstein Theorem, which assures that  if a minimal surface in the Euclidean $3$-dimensional space $\mathbb E^3$ is an entire graphic of a function $f:\Omega \subset \mathbb R^2 \to \mathbb R$, then it is a plane.  Or equivalently, for $S$ being a regular surface of $\euclidean$ and for a fixed 
direction $\rm{span \{\partial_{3}}\}$ and a system of coordinates $(O,x,y,z)$, such that in those coordinates 
 $\partial_{3} = (0,0,1)$, the  Bernstein Theorem  (\cite{1}) assures that 
{\em If $S$ is a minimal surface and the orthogonal projection in the coordinate plane 
$(O,x,y)$ is 1-1 and onto, then the surface is a plane.}

\vspace{0.1cm}
In the context lorentzian, it is well known the  Cabali-Bernstein Theorem which establishes  that in the $3$-dimensional Minkowski space $\mathbb R^3_1$ the only entire minimal graphic  $\{(f(x,y), x, y)| (x,y) \in \Omega \subset \mathbb R^2\}$ are the spacelike planes. One can see the E. Calabi work in \cite{2} as a transposition of the Bernstein Theorem for  $\mathbb R^3_1$, where 
the fixed direction is a timelike unit vector. 

\vspace{0.1cm}
After the Bernstein and Calabi-Bernstein results, several authors have shown interest in these global results, and hence in the literature are found several works  proving the Bernstein property from different viewpoints, providing diverse extensions or new proofs of those theorems. 

Although in codimension one the Bernstein property is hold, it is worth pointing out that the property may be not hold  in codimension bigger that one. That is the case in codimension two, where the Kommerell work (\cite{5}) considers minimal surface in the Euclidean $4$-dimensional space $\mathbb R^4$, 
and  proves that graphic of entire holomorphic function gives minimal surfaces such that its projection in the plane $(O,x,y)$ is 1-1, onto and its Gauss curvature 
is not zero. 

\vspace{0.1cm}

Motived by the results above and on the influence of  the works of J. C. C. Nitsche (\cite{7}) and of T. Rad\'o 
(\cite{9}),  we show through of this paper that the Bernstein property does not hold for spacelike surfaces in the 4-dimensional Minkowski space $\mathbb R^4_1$. More than it, in this paper we also provide answers to the question whether it is possible to establish some extension of the Bernstein Theorem for those kind of surfaces.  Since the inner product used in this case is undefined, we need to consider two cases: fixing a timelike plane or a spacelike plane.  So, explicitly, we work on answering the following question, which is a generalization type of the Bernstein and Kommerell theorems:

{\it Are there complex functions, not necessarily holomorphic, defined in the whole plane, whose  graphic spacelike surface associated to orthogonal projection on a timelike plane or on a spacelike plane 
in $\mink$,  is onto and with non-zero Gauss curvature?} 
\vspace{0.1cm}

Through of this paper we answer the previous question. In fact, we obtain two classes of minimal entire graphic spacelike surfaces in $\mathbb R^4_1$ of type $(A(x,y), x, y, B(x,y))$ and 
$(x, A(x,y), B(x,y), y)$,  for $A(x,y), B(x,y)$ being smooth functions to real-valued, for which there exist points with non-zero Gauss curvature. We call the graphics above, as the first and second type, respectively. Our technique involves the use of a Weierstrass representation involving three
holomorphic functions a complex-valued $a(w), b(w)$ and $\mu(w)$.   That representation allows us to show that for getting the graphic minimal spacelike surfaces the holomorphic functions $a$ and $b$ have to be proportional complexes  if the
graphic is of first type, or inversely proportional complexes with the imaginary part different from zero if the graphic is the second type. Moreover, if the functions $a$ and $b$ are assumed to be defined in whole the complex plane $\mathbb C$, we find classes of graphic surfaces  of first and second type which are entire and minimal with Gauss curvature different from zero. Therefore our theorems  \ref{326} and \ref{500} provide explicit examples which prove that the Bernstein property does not hold in general for spacelike surfaces in $\mathbb R^4_1$. 

\vspace{0.2cm}

Carrying out our study of the spacelike surfaces in $\mathbb R^4_1$, we also obtain explicitly the conjugate minimal spacelike surface 
using the Weierstrass representation. In addition, we identify under what conditions we can guarantee that a non-isothermic 
neighborhood can be extended to the entire complex plane. That is done using the generalized Cauchy-Riemman equations on 
neighborhood in non-isothermic coordinates. So, our work can be seen as an extension of the program developed by T. Rad\'o in \cite{9}. 

In this paper we also give  several examples of graphic minimal spacelike surfaces in $\mathbb R^4_1$ with Gauss curvature non-zero, 
and we find conditions to construct graphic minimal spacelike surfaces 
which have a new type of singularities, it called lightlike  singularities, as 
defined by Kobayashi in \cite{4}.  Those singularities are points where the tangent plane of the surface is also tangent
to the lightcone of $\mink$.

Finally, in the last section of this paper, we construct a $\theta$-family of minimal spacelike surfaces in $\mathbb R^4_1$ which transports a minimal first type graphic surface in $\mathbb E^3$ to a associated minimal first type graphic surface in $\mathbb L^3$. That allows us to conclude, as expected, that the Bernstein Theorem
holds if and only if the Calabi Theorem holds.

\vspace{0.1cm}
For obtaining our results, we use the integral representation of the spacelike surfaces in $\mathbb R^4_1$, and of the adaptation of \cite{3} to the 
Minkowski space $\mink$. The details of this adaptation can be found in the article of authors M.P. Dussan, A. P. Franco Filho and P. Sim\~oes (\cite{DPS}).
Moreover, we pay attention to the Kobayashi work  in \cite{4}, where he used the technique of Weierstrass representation to find several examples of minimal spacelike surfaces $\mathbb R^3_1$ and to find  new type of singularities for these surfaces. Those singularities are points where as manifold these surfaces are defined but where the metric vanishes. That means in those points the tangent plane of $S$ is also tangent to the lightcone of $\mathbb R^3_1$. The Helicoid is a beautiful example that we can find in \cite{4}. 
 
\section{Basic Facts and Notations} 
The Minkowski space $\mink$ will be the $4$-dimensional real space $\real^{4}$ equipped with the bilinear form called of 
Lorentzian product, which is given by 
$$\lpr{(a,b,c,d)}{(t,x,y,z)} = -a t + b x + c y + d z.$$ 

A spacelike plane $V \subset \mink$ is a $2$-dimensional vector subspace where $\lpr{v}{v} > 0$ for each $v \neq 0$ 
of the plane $V$. A timelike plane $T \subset \mink$ is a $2$-dimensional vector subspace where there exists a timelike vector 
$t \in T$, that means that $\lpr{t}{t} < 0$, and an other spacelike vector $n \in T$ such that $\lpr{n}{n} > 0$ with 
$\lpr{t}{n} = 0$. 

We say that timelike plane $T$ is the orthogonal complement of the spacelike plane $V$,  denoted by the symbols $V = T^{\perp}$ and 
$T = V^{\perp}$,  if
$$\lpr{x}{y} = 0 \; \; \mbox{ for all } \; \; x \in T \; \; \mbox{ and } \; \; y \in V.$$   

The following proposition is very useful throughout this work, it establishes a special orthonormal basis for each 
timelike plane. We denote by $\partial_0$ the vector $(1,0,0,0)$. 

\begin{prop}
For each spacelike plane $V \not\subset \euclidean$ equipped with a orthonormal basis $\{e_{1},e_{2}\}$, the (unique) 
timelike plane $T = V^{\perp}$ has an orthonormal basis $\{\tau,\nu\}$ satisfying the following conditions: 

1. \ $\lpr{\tau}{\tau} = -1$ and $\lpr{\tau}{\partial_{0}} < 0$. That means $\tau$ is timelike, future directed 
unit vector of $T$. 

2.\  $\lpr{\nu}{\nu} = 1$ with $\lpr{\nu}{\partial_{0}} = 0$. That means that $\nu$ is a vector into the $3$-dimensional subspace 
$\{0\} \times \real^{3} \subset \mink$, which we will identify with the Euclidean $3$-dimensional vector space $\euclidean$. 

3. $\lpr{\tau}{\nu} = 0$ and for all other orthonormal basis $\{t,n\}$ of $T$ we have that $\tau^{0} \leq \vert t^{0} \vert$. 

4. The orthonormal basis $\{\tau, e_{1}, e_{2}, \nu\}$, in this order, is positive relative to the Minkowski referential 
$\{\partial_{0}, \partial_{1}, \partial_{2}, \partial_{3}\}$ given by the canonical basis of $\real^{4}$. 
\end{prop}

\begin{proof}
We need to define the vector $\tau$, therefore all the statements of the proposition follow immediately. In fact, we take $\tau$ as being
\begin{equation}
\tau = \frac{1}{\sqrt{1 + (e^{0}_{1})^{2} + (e^{0}_{2})^{2} \;}}(\partial_{0} + e^{0}_{1} e_{1} + e^{0}_{2} e_{2}),
\end{equation} 
where $e^{0}_{i} = - \lpr{\partial_{0}}{e_{i}}$ for $i = 1,2$.  It is trivial to see that $\lpr{\tau}{\tau} = -1$, $\lpr{\tau}{e_{i}} = 0$ for $i = 1,2$, and that $\tau$ is 
directed future. Since by the assumption $V \not\subset \euclidean$, we have the timelike plane generated by 
$\{\partial_{0}, \tau\}$. Then we take  $\nu$ to be the unique unit vector of the line $\mathrm{span}\{\partial_{0}, \tau\} \cap \euclidean$ 
such that $\{\tau, e_{1}, e_{2}, \nu\}$ is a positive basis. 

Now, assuming that we have other orthonormal basis $\{t,n\}$ for $T$ we can take the Lorentz transformation given by 
$$t = \cosh \varphi \; \tau + \sinh \varphi \; \nu \; \; \;  \mbox{ and } \;  \; \; n = \sinh \varphi \; \tau + \cosh \varphi \; \nu,$$ 
assumed that $t^{0} > 0$. Since  $-\lpr{t}{\partial_{0}} = - \cosh \varphi \lpr{\tau}{\partial_{0}}$  
it follows then  that $t^{0} > \tau^{0}$.  
\end{proof} 

\subsection{A Semi-rigid frame} Let $\mink = E \oplus T$ be given by the directed sum of a spacelike plane $E$ and its orthogonal 
complement $T$, which is a timelike plane. 

\begin{dfn}\label{def}
A semi-rigid referential of the Minkowski space $\mink$ associated to a directed sum $\mink = E \oplus T$, is a positive 
basis $\{l_{0},e_{1},e_{2},l_{3}\}$ of $\mink$ satisfying the following conditions: 

1. \ $E = \mathrm{span}\{e_{1},e_{2}\}$  and  \ $T = \mathrm{span}\{l_{0},l_{3}\}$. 

2. \ $\{e_{1},e_{2}\}$ is an orthonormal basis for $E$. 

3. \  $\{l_{0},l_{3}\}$ is a null basis for $T$ such that $l_{0}^{0} = 1 = l_{3}^{0}$. 
\end{dfn} 

\begin{prop}
If we have two semi-rigid referential
$\{l_{0},e_{1},e_{2},l_{3}\}$ and $\{\tilde{l}_{0},\tilde{e}_{1},\tilde{e}_{2},\tilde{l}_{3}\}$, 
associated to the directed sum $\mink = E \oplus T$ with $T = E^{\perp}$, then 
$l_{0} = \tilde{l}_{0} \; \; \mbox{ and } \; \; l_{3} = \tilde{l}_{3}.$

Therefore the complex numbers given by 
$$a(l_{3}) = \frac{l_{3}^{1} + i l_{3}^{2}}{1 - l_{3}^{3}}  \; \; \; \mbox{ and } \; \; \; 
b(l_{0}) = \frac{l_{0}^{1} + i l_{0}^{2}}{1 + l_{0}^{3}}$$ 
are univocally determined by the directed sum $\mink = E \oplus T$.   
\end{prop}
\begin{proof}
In the Lorentz plane $T$ with induced orientation by $\partial_{0}$, there exists only two independent 
lightlike directions $L_{0}$ and $L_{3}$. Therefore adding the condition $$\lpr{L_{0}}{\partial_{0}} = -1 = \lpr{L_{3}}{\partial_{0}},$$ 
we obtain the unique basis $\{l_{0},l_{3}\}$ for $T$ given by 3. of the Definition \ref{def}.  
\end{proof}

\begin{corol}
The matrix associated to the set of the semi-rigid referentials of the directed sum $\mink = E \oplus T$, is given by 
$$\M(\vartheta) = \left[\begin{matrix}
1 & 0 & 0 & 0 \\ 0 & \cos \vartheta & \sin \vartheta & 0 \\ 0 & - \sin \vartheta & \cos \vartheta & 0 \\ 0 & 0 & 0 & 1
\end{matrix} \right] \ \  \ \  {\rm for} \ \; \; \vartheta \in \real.$$   
\end{corol} 

\vspace{0.2cm}
Moreover, $\M(\vartheta)$ is a $1$-parameter sub-group of the Minkowski group of isometry of $\mink$, and all geometric facts that 
we will see in this work, are invariant by this sub-group. Indeed, we will see that the complex functions $a(l_{3})$ and $b(l_{0})$ 
determine the geometric properties of minimal spacelike surfaces of $\mink$. 

\begin{prop}
The frame associated to the vector subspace $E$ can be taken in terms of $a(p)$ and $b(p)$, namely, 
$$e_{1}(p) = \frac{W(p) + \overline{W(p)}}{2 \vert 1 - a(p) \overline{b(p)} \vert} \; \; \mbox{ and } \; \; 
e_{2}(p) = \frac{W(p) - \overline{W(p)}}{2 i \vert \vert 1 - a(p) \overline{b(p)} \vert \vert},$$
where 
$$W(p) = (a(p) + b(p), 1 + a(p)b(p), i(1 - a(p) b(p)), a(p) - b(p))
$$ 
with  $\lpr{W(p)}{\overline{W(p)}} = 2\vert 1 - a(p) \overline{b(p)} \vert^{2}$
\end{prop}

\vspace{0.3cm}

\subsection{Spacelike Surfaces in $\mink$}  

\begin{dfn}
A spacelike surface $S \subset \mink$ is a smooth $2$-dimensional sub-manifold of the topological real vector space $\real^{4}$ that 
at each point $p \in S$ its tangent plane $T_{p}S$ relative to the lorentz product of $\mink$ is a spacelike plane. 

A spacelike parametric  surface of $\mink$ is a two parameters map $(U,X)$ from a connected open subset $U \subset \real^{2}$ into 
$\mink$, such that the topological subspace $X(U)$ is a spacelike surface. 

Henceforward we will assume that $(X(U),X^{-1})$ is a chart of a complete atlas for a spacelike surface $S$ of $\mink$.    
\end{dfn} 

Let $((x,y),U)$ be a connected and simply connected open subset of the Euclidean plane $\real^{2}$. 
If $X(x,y) = (X^{0}(x,y), X^{1}(x,y), X^{2}(x,y), X^{3}(x,y))$ is a spacelike parametric surface of $\mink$ then, we have 
a metric tensor induced by the lorentzian semi-metric of $\mink$ given by 
$$\mathbf{g} = \sum_{i,j} \lpr{D_{i} X}{D_{j} X} dx^{i} \otimes dx^{j},$$ 
and the second quadratic form of $S = X(U)$ is a quadratic symmetric $2$-form 
$$B = \sum_{i,j} \Psi_{ij} dx^{i} \otimes dx^{j},$$ 
that is given by covariant partial derivative by the formula
$$D_{ij} X - \sum_k \Gamma_{ij}^{k} D_{k} X = \Psi_{ij}.$$  

From the definition of the Christoffel symbols $\Gamma_{ij}^{k}$ it follows that $\lpr{\Psi_{ij}}{D_{k} X} \equiv 0$. Setting 
a pair of pointwise orthonormal vectors for the normal bundle $NS$ given by $\tau(x,y)$ and $\nu(x,y)$, where 
$\tau(x,y)$ is a future directed timelike unit vector and $\nu(x,y)$ is a spacelike unit vector, we can assume that 
$\lpr{\nu(x,y)}{(1,0,0,0)} \equiv 0$. Therefore we have 
$$\Psi_{ij} = h_{ij} \tau + n_{ij} \nu$$ 
where by definition 
$$h_{ij} = - \lpr{D_{ij} X}{\tau} \; \; \mbox{ and } \; \; n_{ij} = \lpr{D_{ij} X}{\nu}.$$     

Since $\dim (N_{p}S) = 2$ we need to define the normal connection for $S$, which is given by a covariant vector 
$\gamma = \sum \gamma_{k} dx^{k}$ where
$$\gamma_{k} = \lpr{D_{k} \tau}{\nu} = \lpr{D_{k} \nu}{\tau}.$$ 

Next we will display this set of structural equations for $S = X(U)$,  equation (2) being the Gauss equation, (3) and (4) corresponding to  Weingarten equations for $S$.  Namely,
\begin{align}
D_{ij} X - \sum_k \Gamma_{ij}^{k} D_{k} X = h_{ij} \tau + n_{ij} \nu \\
D_{k} \tau = \sum_m h_{m}^{k} D_{m} X + \gamma_{k} \nu \\
D_{k} \nu = - \sum_m n_{m}^{k} D_{m} X + \gamma_{k} \tau.
\end{align}

\begin{dfn}\label{1}
We say that the surface $S = X(U)$ is a minimal surface if and only if 
$$H_{S} = \frac{1}{2}\sum \Psi_{ij} g^{ij} = 0.$$ 
The vector field $H_{S}$ is called the mean curvature vector of $S$. 

It follows from equations (2) that an equivalent definition for minimal surfaces is the condition
$$2 H_{S} = \sum g^{ij} (D_{ij} X - \sum \Gamma_{ij}^{k} D_{k} X) = 
(\Delta_{\mathbf{g}} X^{0}, \Delta_{\mathbf{g}} X^{1}, \Delta_{\mathbf{g}} X^{2}, \Delta_{\mathbf{g}} X^{3}) = 0,$$ 
where $\Delta_{\mathbf{g}}$ is the Laplace-Beltrami operator over $S = (U, {\bf g})$.   
\end{dfn}

Next we observe that one can associate a Riemann surface to $S$. In fact,  from the well know theorem which assures that any spacelike surface admits 
an isothermic coordinate atlas, that means, there is a parametrization 
$$f(w) = (f^{0}(w), f^{1}(w), f^{2}(w), f^{3}(w)), \; \; \; w = u + i v \in U' \subset \complex,$$
such that $f(U') \subset S = X(U)$ and the induced metric tensor is $\mathbf{g} = \lambda^{2} dw d\overline{w}$, or more explicitly
$$\lpr{f_{u}}{f_{u}} = \lambda^{2} = \lpr{f_{v}}{f_{v}} \; \; \mbox{ and } \; \; \lpr{f_{u}}{f_{v}} = 0.$$

\vspace{0.2cm}

Since $f_{w} = \frac{1}{2} (f_{u} - i f_{v})$, we extend the bilinear form of $\mink$ to a complex bilinear form over $\complex^{4} \equiv 
\real^{4} + i \real^{4}$, namely, 
$$\lpr{X + i Y}{A + i B} = \lpr{X}{A} - \lpr{Y}{B} + i(\lpr{X}{B} + \lpr{Y}{A}.$$ 
Hence it implies that 
\begin{equation}\label{11}
\lpr{f_{w}}{f_{w}} = 0 \; \; \mbox{ and } \; \; \lpr{f_{w}}{\overline{f_{w}}} = \lpr{f_{w}}{f_{\overline{w}}} = \lambda^{2}/2.
\end{equation}

Now, if we have two isothermic charts $(U',f)$ and $(V,h)$ for $S$ then, when makes sense, the overlapping map 
is a holomorphic function, and so we can see $M = (S,\A)$ as a Riemann surface equipped with the conformal atlas $\A$, and such that
the induced metric tensor $ds^{2} = \lambda^{2}(w) \vert dw \vert^{2}$ is a compatible metric for the Riemann surface $M$. 

\vspace{0.1cm}
Finally, we note that does not exist compact spacelike surfaces in $\mink$, so from now $M$ will be either 
the disk 
$$D = \{z \in \complex : z \overline z < 1\} \; \; \mbox{ that is a hyperbolic Riemann surface},$$ 
or the complex plane $\complex$ which is a parabolic Riemann surface, since  we are assuming that $M$ is a 
connected and simply connected Riemann surface. Moreover, if $h(z(w)) = f(w)$ then from chain rule it follows  
$$f_{w}(w) = h_{z}(z(w)) \frac{dz}{dw}(w) \; \; \mbox{ and } \; \; 
\lpr{f_{w}}{f_{\overline{w}}} = \lpr{h_{z}}{h_{\overline{z}}} \left\vert \frac{dz}{dw} \right\vert^{2}.$$ 

\vspace{0.1cm}
\subsection{A solution for the equations (\ref{11})} Expanding in its coordinates we have that the equation (\ref{11}) becomes 
$$-(f^{0}_w)^{2} + (f^{1}_w)^{2} + (f^{2}_w)^{2} + (f^{3}_w)^{2} = 0.$$ 
Denoting the complex derivate of the components $f^i_w$ by $Z^i$  and assuming that $Z^{1} - i Z^{2} \neq 0$, we have 
$$\frac{Z^{0} - Z^{3}}{Z^{1} - i Z^{2}} \; \frac{Z^{0} + Z^{3}}{Z^{1} - i Z^{2}} = \frac{Z^{1} + i Z^{2}}{Z^{1} - i Z^{2}}.$$ 

Defining by
$$a = \frac{Z^{0} + Z^{3}}{Z^{1} - i Z^{2}}, \; \; \; \; b = \frac{Z^{0} - Z^{3}}{Z^{1} - i Z^{2}}  \; \; \; \;
\mbox{ and } \; \; \; \; \mu = \frac{Z^{1} - i Z^{2}}{2},$$ 
we obtain that the derivate $f_w$ can be represented by   
$$f_{w} = \mu W(a,b) \; \; \mbox{ where } \; \; W(a,b) = (a + b, 1 + ab, i(1 - ab), a - b),$$ 
from $(a,b) \in \F(M,\complex) \times \F(M,\complex)$  in  $\complex^{4}$. 

\vspace{0.1cm}
Moreover, we have that 
$\lambda^{2} = 2 \lpr{f_{w}}{\overline{\; f_{w}}\;} = 4 \mu \overline{\mu} (1 - a \overline{b})(1 - \overline{a} b).$ 
Therefore, $\mu \neq 0$ and $1 - a \overline{b} \neq 0$ are the conditions to obtain a surface without 
singularities in its metric.

\vspace{0.2cm}
Now, since we can write $W(a,b) = (a,1,i,a)+ b(1,a,-ia,-1)$ we obtain the cases where happens $(Z^{1} + i Z^{2})(Z^{1} - iZ^{2}) = 0$ 
through of the expressions $f_{w} = \eta (a,1,i,a)$ and $f_{w} = \xi (1,a,-ia,-1)$. Moreover when $Z^{0} = 0 = Z^{3}$ we obtain 
$f_{w} = \eta(0,1,i,0)$ which can be identified with the plane $\{0\} \times \real^{2} \times \{0\}$. 

\vspace{0.2cm}
The following lemma is an extension to $\mink$ of a theorem obtained by Monge: 
\begin{lemma}\label{2}
For a $\lambda$-isothermic spacelike parametric surface $(U,f)$ the following statement are equivalent: 

(i) The surface $f(U)$ is minimal, $H_{f}(w) \equiv 0$. 

(ii) The maps $\mu,a,b$ are holomorphic functions from $U$ into $\complex$.   
\end{lemma}

\begin{proof}
It follows from the Laplace-Beltrami operator that   
$\Delta_{M} f^{i}(w) = \frac{2}{\lambda^{2}} (f^{i}(w))_{w \overline{w}} = 0$ for $i = 0,1,2,3.$ 
\end{proof}

\subsection{An integral representation} Let $(U,X)$ be a spacelike parametric surface of $\mink$ where 
$$X(x,y) = (X^{0}(x,y), X^{1}(x,y), X^{2}(x,y), X^{3}(x,y))$$ 
and $U \subset \real^{2}$ is a simply connected domain. Then, the vector $1$-form given by 
$$dX = \frac{\partial X}{\partial x} dx + \frac{\partial X}{\partial y} dy$$ 
is exact and therefore closed.  So, the integral equation associated to $(U,X)$ is  
\begin{equation}\label{3}
X(x,y) = X(x_{0},y_{0}) + \int_{(x_{0},y_{0})}^{(x,y)} \frac{\partial X}{\partial x} dx + \frac{\partial X}{\partial y} dy. 
\end{equation}

\vspace{0.2cm}
Moreover, each solution of equation (\ref{3}) is a spacelike  parametric surface $(U,X)$ if it holds  
$$E = \lpr{X_{x}}{X_{x}} > 0, \; \; G = \lpr{X_{y}}{X_{y}} > 0, \; \; F = \lpr{X_{x}}{X_{y}} \; \mbox{ and } \; \; 
EG - F^{2} > 0.$$    

\vspace{0.2cm}

From Definition \ref{1} and Lemma \ref{2},  we obtain: 

\begin{corol}
Let $U \subset \real^{2}$ a simply connected domain.  If $(U,X)$ is a minimal spacelike parametric surface which is  solution 
of the integral equation (\ref{3}), then each coordinate function of $X(x,y)$ is a harmonic real-valued function on $U$. 
\end{corol} 
\begin{proof}
Indeed, the Laplace-Beltrami operator $\Delta_{M}$ is a tensorial operator defined by contraction of the Gauss equation (2), 
as follows in Definition \ref{1}. 
\end{proof}

We note that with isothermic local coordinates the integral representation (\ref{3}) is usually called of Weierstrass integral equation, namely,  
$$f(w) = p_{0} + 2 \Re \int_{w_{0}}^{w} \mu(\xi) W(a(\xi),b(\xi)) d \xi,$$ 
where $f_w(w)$ is the solution of equation (\ref{11}) in Subsection 2.2. 

\subsection{The structural equations with isothermic parameters} Let $(U,f)$ be a parametric sub-surface of $X(M)$ given 
with isothermic parameters $w = u + iv$ such that  $\lpr{f_{w}}{f_{w}} = 0$ and $\lpr{f_{w}}{f_{\overline{w}}} = \lambda^{2}/2$. 
In this case we have the following version of structural equations (2), (3) and (4) for minimal surfaces.
\begin{lemma}
Let $(U,f)$ be a $\lambda^{2}$-isothermic coordinates system for a minimal surface $(M,X)$ of $\mink$. We have the following 
structural equations in $w = u + i v \in U$:
\begin{align}
\tau_{w} = \sigma f_{\overline{w}} + \Gamma \; \nu \; \; \; \; \; \;  \mbox{ and } \; \; \; \;  \nu_{w} = \chi f_{\overline{w}} + \Gamma  \; \tau \\
f_{ww} = 2 \frac{\lambda_{w}}{\lambda} f_{w} + \frac{\sigma \lambda^{2}}{2} \; \tau - \frac{\chi \lambda^{2}}{2} \; \nu \; \; \; \; \; \;  and 
\; \; \; \; f_{w\overline{w}} = 0 \\ \Gamma(w) = \lpr{\tau_{w}}{\nu} = - \lpr{\nu_{w}}{\tau} = \frac{\gamma_{1}(w) - i \gamma_{2}(w)}{2}
\end{align} 
\end{lemma}
\begin{proof}
We start showing equation (8). For that we take $f_{ww} = A f_{w} + B f_{\overline{w}} + C \tau + D \nu$, 
and assume that equations (7) and (9) are the definition of the functions associated to the normal connection for $(U,f)$. 

From $\lpr{f_{w}}{f_{w}} = 0$ it follows $\lpr{f_{ww}}{f_{w}} = 0$, therefore $B = 0$. From 
$\lpr{f_{w}}{f_{\overline{w}}} = \lambda^{2}/2$ it follows $\lpr{f_{ww}}{f_{\overline{w}}} + \lpr{f_{w}}{f_{w\overline{w}}} = 
\lambda_{w} \lambda$, and, since $f_{w\overline{w}} = 0$ we obtain  $A = 2 \frac{\lambda_{w}}{\lambda}.$ 

Now, from $\lpr{f_{w}}{\tau} = 0$ we have that $\lpr{f_{ww}}{\tau} + \lpr{f_{w}}{\tau_{w}} = 0$, therefore we obtain 
$ C = \sigma \frac{\lambda^{2}}{2}$.  Analogously one has  $D =- \chi \frac{\lambda^{2}}{2}.$ So we have showed equation (8). 

\vspace{0.1cm}
The definition of the functions $\sigma$ and $\chi$ is obtained by equations (7), that from $\lpr{f_{\overline{w}}}{\tau} = 0$ 
and from the minimal condition for $(M,f)$ it follows that $\lpr{\tau_{w}}{f_{\overline{w}}} + \lpr{\tau}{f_{w\overline{w}}} = 0$. Thus 
the tangential component of $\tau_{w}$ is $\sigma f_{\overline{w}}$. Then, we take the equations (7) as a definition of the 
functions associated to the normal connection of $(M,f)$. Equation (9) defines the function $\Gamma$. 
\end{proof}

\section{Two types of Graphics for Minimal Surfaces of $\mink$} 

First, let us recall that $\mink$ has topological structure and differential structure of the Euclidean space $\real^{4}$. 

If $R(u,v) = (\varphi(u,v),\psi(u,v))$ is a function from $U \subset \real^{2}$ in $\real^{2}$, we can see as a graphic of $R$ 
the set of point of $\real^{4}$ such that
$$\mbox{graphic(R)} = \{((u,v),(\varphi(u,v),\psi(u,v))) \in \real^{4}: (u,v) \in U \subset \real^{2}\}.$$
Since we can choose four  equivalent positions for the timelike axis in $\mink$,  we only need to pick 
 two of those positions to get all the possibilities of graphic surfaces. In fact: 

\vspace{0.2cm}
Fixing the signature of $\mink$ by $(-1,+1,+1,+1)$ we take by definition: 

\vspace{0.2cm}
(1) The first type of graphic surfaces as given by 
$$X(x,y) = (A(x,y), x, y, B(x,y)) \; \mbox{ where } \; (x,y) \in U \subset \real^{2}.$$ 

(2) The second type of graphic surfaces as given by 
$$X(x,y) = (x, A(x,y), B(x,y), y) \; \mbox{ where } \; (x,y) \in U \subset \real^{2}.$$ 

\vspace{0.1cm}
We will always assume that the functions $A$ and $B$ are $\C^{\infty}(U)$,  $U$ is a connected and simply 
connected open subset of $\real^{2}$ and that $X(U)$  is a spacelike surface of $\mink$. 

\begin{prop}\label{14}
A minimal graphic surface (first or second type) of $\mink$ satisfies the following system of equations 
\begin{equation}\label{10}
\left\{\begin{matrix}
g_{22} D_{11} A - 2 g_{12} D_{12} A + g_{11} D_{22} A = 0 \\ g_{22} D_{11} B - 2 g_{12} D_{12} B + g_{11} D_{22} B = 0 
\end{matrix} \right.
\end{equation} 
where $\mathbf{g} = \sum_{ij} g_{ij} du^{i} du^{j}$ is the  positive defined metric tensor associated to the surface $S = X(U)$.  

The system of equations (\ref{10}) only says that $A$ and $B$ are harmonic functions of the Riemann surface $(U,X)$. 
\end{prop}
\begin{proof}
Taking the matrix representation of metric tensor and its inverse tensor 
$$[g_{ij}] = \left[\begin{matrix} E & F \\ F & G \end{matrix} \right], \; \;  \; \; \; \; \; 
[g^{ij}] = \frac{1}{EG - F^{2}} \left[\begin{matrix} G & -F \\ -F & E \end{matrix} \right],$$ 
one has, from Definition \ref{1}, that the mean curvature vector is given by 
$$2 H_{X} = \frac{1}{EG - F^{2}} (G \Psi_{11} - 2F \Psi_{12} + E \Psi_{22}).$$ 

Now, for each type of surface we take a pointwise basis $\{N_{1},N_{2}\}$ for its normal bundle, as follows.

If $X(x,y) = (A(x,y), x, y, B(x,y))$ we take the orthogonal vectors 
$$N_{1} = (1, A_{x}, A_{y}, 0) \; \; \mbox{ and } \; \; N_{2} = (0, B_{x}, B_{y}, -1),$$ 
and so in this case, $D_{ij}X = (D_{ij} A, 0 ,0, D_{ij} B)$. 

\vspace{0.2cm}
If $X(x,y) = (x, A(x,y), B(x,y), y)$ we take the orthogonal vectors 
$$N_{1} = (A_{x},1,0,-A_{y}) \; \; \mbox{ and } \; \; N_{2} = (B_{x},0,1,-B_{y}).$$
Then in this case $D_{ij}X = (0,D_{ij} A, D_{ij} B, 0)$.
Therefore, the system (\ref{10}) follows immediately.
\end{proof}

Our first example corresponds to minimal  spacelike surfaces, which are  graphic surfaces of the first type
 defined in the whole plane $\real^{2}$. 

\begin{example}\label{323}
For each harmonic function $\theta : \real^{2} \longrightarrow \real$ the maps
$$X(x,y) = (\theta(x,y), x,  y, \theta(x,y)) \; \; \mbox{ or } \; \; X(x,y) = (\theta(x,y), x,  y, -\theta(x,y))$$ 
are both minimal spacelike parametric surfaces, locally isometric to the Euclidean plane $\real^{2}$, 
and therefore flat surfaces. 

In fact, it assumes the first expression of $X(x,y)$.   
Since $X_{x} = (\theta_{x}, 1, 0, \theta_{x})$ and $X_{y} = (\theta_{y}, 0, 1, \theta_{y})$ it follows $\lpr{X_{x}}{X_{x}} = 1 = \lpr{X_{y}}{X_{y}}$ with $\lpr{X_{x}}{X_{y}} = 0$. 
Now, by assumption $\Delta \theta = \theta_{xx} + \theta_{yy} = 0$, it follows that $H_{X}(x,y) = (0,0,0,0)$. 

\vspace{0.1cm}
We also observe that, according the notation of Subsection 2.3, this class of surfaces corresponds to when $Z^1 + i Z^2 =0,$ with $Z^0 - Z^3 =0$ and $Z^0 \ne 0$, where $Z^i$ are the components in the representation $X_w(w) =  (\theta_w, \frac{1}{2}, \frac{i}{2},  \theta_w).$ 

Moreover we can write these parametric surfaces as follows:  For $A = B = \theta(x,y)$ we have that $X(x,y) = (0, x, y, 0) + \theta(x,y) (\partial_{0} + \partial_{3})$, therefore 
$X(\real^{2})$ is a subset of a degenerated hyperplane, and this shows that its normal curvature vanishes identically.
\end{example}  
 
The Example \ref{323} shows that we need a formula of the second quadratic form in terms of functions $\mu$, $a$ and $b$. That formula was already obtained in 
Theorem 3.3 from \cite{DPS}, so we rewrite next. 
\begin{lemma}\label{2.99}
Let $f_{w} = \mu W(a,b)$, where $a$ and $b$ are holomorphic functions from $M$ into $\complex$. The second quadratic form in 
complex notation is given by
\begin{equation}\label{2.9}
(f_{ww})^{\perp} = \frac{\mu a_{w}}{1 - a \overline{b}} L_{0}(b) + \frac{\mu b_{w}}{1 - b \overline{a}} L_{3}(a),
\end{equation} 
where $L_0(b)$ and $L_3(a)$ are  future directed lightlike vectors given by
$$
L_0(b) = (1+ b\overline b, b + \overline b, -i(b-\overline b), 1- b \overline b)
 \ \ \ {\rm and} \ \ \  L_3(a) = (1+ a\overline a, a + \overline a, -i(a-\overline a), - 1+ a \overline a).
$$
\end{lemma}

It follows from Lemma \ref{2.99} the next corollary.

\begin{corol}
The second quadratic form of a minimal spacelike  surface $(U,f)$ is lightlike type if and only if $a_{w} = 0$ or $b_{w} = 0$. 
Therefore in this case, the Gauss curvature $K(f) = 0$ and the surface is contained in a  degenerated hyperplane. 

Reciprocally, if the Gauss curvature $K(f) = 0$ then the second quadratic form is lightlike type or it is zero, 
$(f_{ww})^{\perp} = 0$.
\end{corol}

Now, we apply  equations (\ref{10})  for graphic minimal surfaces in $\euclidean$ and $\submink$. We give 
the explicit equation for each case. 

For the first type:   

(1) When $A(x,y) \equiv 0$ we obtain the graphics in $\euclidean$ given by an unique function $B(x,y)$: 
$$f(x,y) = (0,x,y,B(x,y)) \in \euclidean,$$ 
with the induced metric tensor over $f(U)$ as a spacelike surface of $\mink$. Then system (\ref{10}) becomes to the equation 
\begin{equation}\label{4}
(1 + B_{y}^{2}) B_{xx} - 2 B_{x} B_{y} B_{xy} + (1 + B_{x}^{2}) B_{yy} = 0,
\end{equation}     
which is called the equation of minimal graphic for smooth surface of the Euclidean space $\real^{3}\equiv \mathbb E^3$.   
In this case Bernstein showed that if $U = \real^{2}$ then the solution of equation (\ref{4}) is a plane. 

\vspace{0.2cm}
(2) When $B(x,y) \equiv 0$ we obtain the graphics in $\euclidean$ given by an unique function $A(x,y)$: 
$$f(x,y) = (A(x,y),x,y,0)) \in \submink,$$ 
with the induced metric tensor over $f(U)$ as a spacelike surface of $\mink$.  System (\ref{10}) becomes to the equation
\begin{equation}\label{5}
(1 - A_{y}^{2}) A_{xx} + 2 A_{x} A_{y} A_{xy} + (1 - A_{x}^{2}) A_{yy} = 0 \; \; \; \mbox{ with } \; 
A_{x}^{2} < 1 \; \mbox{ and } \; A_{x}^{2} + A_{y}^{2} < 1, 
\end{equation}     
which is called the equation of minimal graphic for smooth surface of the Lorentzian space $\submink$.  For this case,  
Calabi showed that if $U = \real^{2}$ then the solution of equation (\ref{5})  is a plane. 

\vspace{0.3cm}
Now we turn our attention for graphic minimal spacelike surfaces of the second type,  given by the representation  $f(x,y) = (x, A(x,y),B(x,y), y)$.  In this case,  

\vspace{0.1cm}
(3) When $B(x,y) \equiv 0$ we obtain the graphics given by an unique function $A(x,y)$: 
$$f(x,y) = (x,A(x,y),0, y)) \in \submink,$$ 
with the induced metric tensor over $f(U)$ as a spacelike surface of $\mink$. Then system (\ref{10}) becomes to the equation
\begin{equation}
(1 + A_{y}^{2}) A_{xx} - 2 A_{x} A_{y} A_{xy} + (-1 + A_{x}^{2}) A_{yy} = 0 \; \; \; \mbox{ with } \; 
A_{x}^{2} > A_{y}^{2} + 1, 
\end{equation}     
and, we will say that this equation is the equation for graphic of second type of minimal smooth 
surface of $\submink$.

\section{About the Extension of Local Solutions of the Graphic Equations} 
In this section we study whether it is possible to extend to whole the complex plane $\mathbb C$ the local solutions for the graphic equations given in 
system (\ref{10}). 

\vspace{0.2cm}
We start identifying a formula for the Gauss curvature  of the surface. In fact, for $f_{w} = \mu W(a,b)$ where $(U,f)$ is a minimal spacelike surface of $\mink$, 
with holomorphic functions $a(w)$, $b(w)$, $\mu(w)$, we know that the expression for the Gauss curvature is given by 
$$K(f) = - \frac{ \Delta \ln \lambda^{2}}{2\lambda^{2}} = - \frac{1}{\lambda^{2}} \Delta \ln \lambda.$$

Now, since $\lambda^{2} = 4 \mu \overline{\mu} (1 - a \overline{b})(1 - \overline{a}b)$ and $\Delta  = 4 \partial_{w \overline{w}},$ 
we obtain
$$K(f) = - \frac{(\ln (1 - a \overline{b})(1 - \overline{a}b))_{w\overline{w}}}{2\mu \overline{\mu} (1 - a \overline{b})(1 - \overline{a}b)}.$$ 
Since $$(\ln(1 - a \overline{b})(1 - \overline{a}b))_{w\overline{w}} = 
-a_w\left(\frac{\overline{b}}{1 - a \overline{b}}\right)_{\overline{w}} - 
b_w\left(\frac{\overline{a}}{1 - b \overline{a}}\right)_{\overline{w}},$$
it follows that 
\begin{equation}
K(f) = \frac{\Re(a_{w}\overline{b}_{\overline{w}}(1 - \overline{a}b)^{2})}
{\mu \overline{\mu} (1 - a \overline{b})^{3}(1 - \overline{a}b)^{3}}.
\end{equation} 

{\bf First case}. We will focus our attention to find surfaces given by 
$$X(x,y) = (A(x,y), x, y, B(x,y)) \; \; \mbox{ for all } \; \; (x,y) \in \real^{2},$$ 
satisfying the equations (10), which means that $X(\real^{2}) = S$ is a minimal surface of $\mink$.  

\vspace{0.2cm}
So a question arises: {\em Is there a non-flat solution to this problem}? 

\vspace{0.2cm}
For answering that question we proceed as follows. First, we construct a pointwise basis for the normal bundle.
 In fact,
it takes the vector fields $N_1$ and $N_2$,  along $S = X(\real^{2})$, used in the proof of Proposition \ref{14}, namely, 
$$N_{1} = (1, A_{x}, A_{y}, 0) \; \; \mbox{ and } \; \; N_{2} = (0, -B_{x}, -B_{y}, 1).$$ 

Then we have the following proposition.

\begin{prop}
The spacelike Gauss map $\nu(x,y)$ for the minimal surface $S \subset \mink$ is given by 
$$\nu(x,y) = \frac{1}{\sqrt{1 + (B_{x})^{2} + (B_{y})^{2}}} (0, - B_{x}, - B_{y}, 1).$$
\end{prop}    
\begin{proof}
We only need to see if the orientation of $\{N_{1},N_{2}\}$ and the orientation of $\{\partial_{0},\partial_{3}\}$ 
are compatible each other. The compatibly orientations follow from the projected vectors  
$(N_{1}^{0},0,0,N_{1}^{3}) = \partial_{0}$ and $(N_{2}^{0},0,0,N_{2}^{3}) = \partial_{3}$.
\end{proof}

\begin{corol}
The Gauss map $\nu : S \longrightarrow \mathbb{S}^{2} \subset \euclidean$ is such that 
$$\nu^{3} = \frac{1}{\sqrt{1 + (B_{x})^{2} + (B_{y})^{2}}} > 0.$$ 

In other words, $\nu(S)$ is the (open) north hemisphere of the Riemann sphere $\mathbb{S}^{2}$.  
\end{corol}

 Now we assume that we have a local representation $(U,f)$ such that $f(U) \subset S$ and 
$$f_{w} = \mu(a + b, 1 + ab, i(1 - ab), a - b),$$ 
where $a,b,\mu$ are holomorphic functions from $U$ into $\complex$, and $U$ is a connected and 
simply connected open subset of $\complex$. Then the normal bundle has a pointwise basis
 of lightlike vectors  $\{L_{3}(a), L_{0}(b)\}$ like in Lemma 3.2, which allows, in easier form, to compute the 
fourth component of the spacelike Gauss map $\nu(a,b)$, as follows.

\begin{lemma}\label{111}
For an isothermic local representation $(U,f)$ such that $f(U) \subset S$ we have 
\begin{equation}\label{3.9.9}
\nu^{3}(a,b) = \frac{1}{\vert 1 - \overline{a}b \vert \sqrt{1 + \vert a \vert^{2}} \; \sqrt{1 + \vert b \vert^{2}}}
(1 - \vert ab \vert^{2}).
\end{equation} 

Moreover, the maximal extension of holomorphic functions $a,b$, is conditioned by the inequalities:  
\begin{equation}\label{2.8.8}
\vert 1 - \overline{a}b \vert \neq 0 \; \; \mbox{ and } \; \; \vert ab \vert^{2} \neq 1.
\end{equation}
\end{lemma}
\begin{proof}
Taking the normalization of the vector $N_3$ given by 
$$N_{3}(a,b) = \frac{1}{1 + b \overline{b}} L_{0}(b) - \frac{1}{1 + a \overline{a}} L_{3}(a)$$ 
one gets $\nu(a,b)$ since $N_{3}^{0}(a,b) = 0$. Therefore, we obtain the component $\nu^{3}(a,b)$ given in (\ref{3.9.9}) and the inequalities (\ref{2.8.8}).
\end{proof}

\vspace{0.1cm}
Now we observe that the first inequality in (\ref{2.8.8}) is the functional area $\sqrt{EG - F^{2}} = \vert \mu \vert \; \vert 1 - \overline{a}b \vert$. 
Then for our purposes, we will find a necessary and sufficient condition to obtain a maximal extension of the function $\sqrt{EG - F^{2}}$. Hence if we assume the integrating factor being constant $\mu =1$, we need just to consider the maximal extension of   $\vert 1 - \overline{a}b \vert$.

 For achieving that goal we give the next corollary, which follows from Liouville Theorem and Theorem \ref{111}, since for $a(w) b(w)$ being an entire bounded function, it must be constant. 

\begin{corol}\label{27} If $a(w)$ and $ b(w)$ can be extended for whole the plane $\complex$, then there exists a constant $c \in \complex$ 
such that $a(w) b(w) = c$.  
\end{corol}

Hence it follows as direct consequence of Corollary \ref{27}, that if $a(w) = b(w)$ or $a(w) = - b(w)$ for all $w \in \complex$, then $a(w) = \sqrt{c}$.  That means that $(\complex,f)$ is a spacelike plane of $\mink$.      

\vspace{0.1cm}

Moreover from the Corollary \ref{27}, we can also construct an example of a minimal surface $(\complex,f)$, which is a graphic 
with Gauss curvature $K(f) \neq 0$. This means a set of points $p$ of the surface such that the condition  $K(p) = 0$ is not satisfied on the entire plane $\complex$. 
Even more, now we are abled to prove our next  result which provides a general class of examples 
of  entire graphic minimal surfaces of first type such that the Gauss curvature $K(f) \neq 0$.

\begin{theor}\label{326}
Let $a = a(w)$ be a holomorphic function defined in whole the plane $\complex$ such that $a(w) \neq 0$ for each $w \in \complex$. Let 
$c = \alpha + i \beta \in \complex \setminus \{0,1,-1\}$ such that $\alpha^{2} + \beta^{2} \neq 1$, and it takes the holomorphic function $b(w) = \frac{c}{ a(w)}$ from $\complex$ in $\complex$.  Then the surface given by  
\begin{equation}\label{29} 
f(w) = X_{0} + 2 \Re \int_{0}^{w} \left(a(\xi) + \frac{c}{a(\xi)}, 1 + c, i(1 - c), a(\xi) - \frac{c}{a(\xi)}\right) d \xi,
\end{equation} 
is a minimal surface of $\mink$, which is a graphic surface of type $X(x,y) = (A(x,y), x, y, B(x,y))$ through of the transformation of coordinates 
given by $x_{w} = (1 + c)$ and $y_{w} = i(1 - c)$. 

Moreover, assuming that $a(w)$ is not a constant function then, there exists a point $p \in S$ such that $K(p) \neq 0$. Hence the surface can not be contained in hyperplanes of $\mink$.   
\end{theor}
\begin{proof}
Taking $x(u,v) =  2[(1 + \alpha)u - \beta v]$ and $y(u,v) = 2[ \beta u + ( \alpha -1) v]$ we get the equation of the coordinates change, namely, 
\begin{equation}\label{30}
\left[\begin{matrix} u \\ v \end{matrix} \right] =  
\frac{1}{2[ \alpha^{2} + \beta^{2} -1]} \left[\begin{matrix} \alpha -1 & \beta \\ - \beta & 1 + \alpha \end{matrix} \right] \; 
\left[\begin{matrix} x \\ y \end{matrix} \right].
\end{equation}  

Therefore, since $a(w)$ and $b(w)$ are holomorphic functions and  $\alpha^{2} + \beta^{2} \neq 1$, we  obtain that equation (\ref{29}) represents a graphic minimal surfaces of first type. 

\vspace{0.1cm}
Since the metric is given by $\lambda^{2} = 4 \vert 1 - \overline{a} c/a \vert^{2}$ follows that $\Delta \ln \lambda \neq 0$ 
in points where $a_w(w) \neq 0$. Then, since $K(f) = - \frac{1}{\lambda^2}\Delta \ln  \lambda$,  it follows that in those points happen $K(f) \neq 0$. 

Next, by integration we can obtain the components functions $A(w) = f^{0}(w)$ and 
$B(w) = f^{3}(w)$, and through of the coordinate transformation given by the equation (\ref{30}) we obtain the explicit representation as graphic surface. 

\vspace{0.1cm}
To finish, we see the real spacial property of surface $S$. In fact, it supposes that there is a vector $v = (v^{0},v^{1},v^{2},v^{3}) \in \mink$ 
such that $\lpr{v}{f_{w}} = 0$. Then from the equality $-v^{0}(a + b) + v^{1}(1 + ab) + iv^{2} (1 - ab) + v^{3}(a - b) = 0$, we obtain 
$$(v^{3} - v^{0}) a - (v^{3} + v^{0}) b + (v^{1} + i v^{2}) + ab (v^{1} - i v^{2}) = 0.$$ 
It defines $T = v^{3} - v^{0}$, $S = v^{3} + v^{0}$, $Z = v^{1} + i v^{2}$, then we obtain 
$(T a + Z) + b(a \overline{Z} - S) = 0,$ which implies that 
$$b = \frac{ Ta + Z}{S - a \overline{Z}} = \frac{c}{a} \; \; \mbox{ if and only if } \; \; T = 0 = S \; \; \mbox{ and } \; \; 
c = - \frac{Z}{\overline Z}.$$ 

Thus,  from $\frac{Z}{\overline Z} = -c$ and $v^{0} - v^{3} = 0 =  v^{0} + v^{3}$, it follows that $v \not\in \mink$. Contradiction. 
\end{proof}

So from Theorem \ref{326} we can construct a classe of minimal graphic surfaces of first type, whose Gauss curvature is not null in some
points of the surface. That means the classic Bernstein theorem does not hold in this case. Next we give some particular examples of that
fact.

\begin{example}
For a simple example, we take $a = e^{w}$ and $c = 2$. Then according to Theorem \ref{326} we can take 
$b = \frac{c}{a}= 2 e^{-w}$ and $X_0 = 2(-1,0,0,3)$, to have the parametrization 
$$
f(w) = 2((e^u - 2 e^{-u}) \cos v, 3u, v,  (e^u + 2 e^{-u}) \cos v).
$$
Therefore taking the coordinates transformation given by $ x = 6u$ and $y = 2v$, we get the graphic parametrization given by
$$
X(x,y) = (2 (e^{\frac{x}{6}} - 2 e^{- \frac{x}{6}}) \cos (\frac{y}{2}), \; x, \;  y, \; 2 (e^{\frac{x}{6}} + 2 e^{- \frac{x}{6}}) \cos (\frac{y}{2})),
$$
for which there are points such that the Gaussian curvature is not zero. In fact, it is just to take $\alpha$ and $\beta$ such that
$\alpha \ne \cos (2y)$ and $\beta \ne \sin(2y)$, that means, $c \ne e^{2iy}$.
\end{example}

\begin{example}\label{90}
In this example we use Theorem \ref{326} to construct minimal graphic surfaces of first type. 
We start assuming   $a(w) = e^{w}$ and $b(w)  =\frac{2 e^{i\theta}}{a(w)}$  for  $\theta \in (0,\pi)$. 
Since $\vert c \vert = \vert 2 e^{i\theta} \vert = 2$, the condition  $\alpha^2 + \beta^2 \ne 1$ is hold.  Then $W(a,b)$ is given by  
$$W(a,b) = (e^{w} + 2 e^{i\theta} e^{-w}, 1 + 2 e^{i\theta}, i(1 - 2 e^{i\theta}), e^{w} - 2 e^{i\theta}e^{-w}).$$ 
Now we take the factor of integration $\mu = 1$, to obtain the integral representation (\ref{29}) given by 
$$f(w) = 2 \Re \int_{0}^{w}(e^{\xi} + 2 e^{i\theta} e^{-\xi}, 1 + 2 e^{i\theta}, i(1 - 2 e^{i\theta}), 
e^{\xi} - 2 e^{i\theta}e^{-\xi}) d \xi,$$ 
more explicitly
\begin{equation}\label{800}
f(u,v) = 2(e^u \cos v - 2 e^{-u} (\cos v \cos \theta + \sin v \sin \theta), (1+ 2\cos \theta) u - 2 v \sin \theta,  
\end{equation}
$$
(-1+ 2 \cos \theta) v + 2 u \sin \theta, e^u \cos v + 2 e^{-u}( \cos v \cos \theta + \sin v \sin \theta)).
$$
Hence making the coordinates transformation $x_w = 1+ 2 e^{i\theta}$ and $y_w = i(1-2 e^{i\theta})$, we get
$$x = 2[ (1+2 \cos \theta) u - 2 v \sin \theta] \ \ \ \ \  {\rm and} \ \ \ \ \ y = 2[ (-1+ 2 \cos \theta) v + 2 u \sin \theta].
$$ 
Thus the minimal graphic surface is given by $X(x,y) = (A(x,y), x, y, B(x,y))$, where the functions $A(x,y), B(x,y)$ are given by the first and fourth component of
formula (\ref{800}) with
$$
u= \frac{1}{6}((2 \cos \theta -1) x + 2 y \sin \theta) \ \ \ \ \  {\rm and} \ \ \ \  \  v = \frac{1}{6}(-2 x \sin \theta  + (1+ 2 \cos \theta)y).
$$ 
We observe that since $a_w = e^w$ never vanishes, all the points of the graphic surface are such that $K(p) \ne 0$.

\end{example}

{\bf Second case}. We will focus our attention to find surfaces given by 
$$X(x,y) = (x, A(x,y), B(x,y), y) \; \; \mbox{ for all } \; \; (x,y) \in \real^{2},$$ 
satisfying the equations (10). That means that $X(\real^{2}) = S$ is a graphic minimal surface of $\mink$ of second type.

\vspace{0.2cm}
So a question arises: {\em Is there a non-flat solution to this problem}? 

\vspace{0.2cm}
For answering that question we proceed as before, constructing first a pointwise basis for the normal bundle.

 Let us take the attitude matrix of $dX$: 
$$[dX]^{t} = \left[\begin{matrix}
1 & A_{x} & B_{x} & 0 \\ 0 & A_{y} & B_{y} & 1
\end{matrix}\right].$$

The unit spacelike Gauss map $\nu = \nu(x,y)$ is given by 
$$\nu(x,y) = \frac{1}{\sqrt{J^{2} + (B_{x})^{2} + (A_{x})^{2}}}(0, B_{x}, -A_{x}, J) \; $$
for  $J = \frac{\partial(A,B)}{\partial(x,y)} = A_{x}B_{y} - A_{y} B_{x}.$

Since we can not control the functions $\nu^{i}$ for $i = 1,2,3$, we will work with the Weierstrass form 
$$f_{w} = \mu (a + b, 1 + ab, i(1 - ab), a - b)$$ 
and the transformation of coordinates 
\begin{equation}\label{299}
x_{w} = \mu (a + b) \; \mbox{ and } \; y_{w} = \mu( a - b), \; \mbox{ where } \; x_{w}y_{\overline{w}} - x_{\overline{w}}y_{w} = 
2 \vert \mu \vert^{2} \; (\overline{a} b - a \overline{b}).
\end{equation} 

\begin{lemma}\label{200}
It considers the  transformation of coordinates given by equations (\ref{299}). Then Jacobian function 
$x_{w}y_{\overline{w}} - x_{\overline{w}}y_{w} = 2 \vert \mu \vert^{2} \; (\overline{a} b - a \overline{b})$ 
does not vanish in a domain $U \subset M$ if and only if, for each $w \in U$, 
\begin{equation}\label{277}
a(w) \ne 0 \ne b(w) \ \ \ \ {\rm and} \ \ \ \ \ \   \Im(\frac{a(w)}{b(w)}) \neq 0.
\end{equation} 

A maximal extension of holomorphic functions $a,b$ is conditioned  by the 
inequalities (\ref{277}) and by $\vert 1 - \overline{a}b \vert \neq 0$.
\end{lemma}
\begin{proof}
First we observe that $a(w) \neq 0 \neq b(w)$ is a necessary condition. Moreover, for each $w \in U$, 
$$-2i \Im(\frac{a(w)}{b(w)} )= \frac{\overline{a(w)}}{\overline{b(w)}} - 
\frac{a(w)}{b(w)} = \frac{\overline{a(w)} b(w) - a(w) \overline{b(w)}}{b(w) \overline{b(w)}}.
$$
Hence, since the Jacobian function does not vanish, it follows that $ \Im(\frac{a(w)}{b(w)} ) \ne 0$. The conversely follows
 immediately. 
\end{proof}

From Lemma \ref{200}  and from Little Picard Theorem, it follows the next corollary.

\begin{corol}\label{300}
It assumes that the holomorphic functions $a(w)$ and $b(w)$ can be extended for whole the plane $\complex$. Then
 there exists a constant $c \in \complex \setminus \{0,1,-1\}$ such that $b(w) = c a(w)$. 

Moreover, as consequence, if $f_w$ is such that $f_{w} = \mu (a(1 + c), 1 + ca^{2}, i(1 - ca^{2}), a(1 - c))$ then  
$$x(w) = 2 \Re \left((1 + c) \int_{0}^{w} \mu(\xi) a(\xi) d \xi\right) \; \; \mbox{ and } \; \; 
y(w) = 2 \Re \left((1 - c) \int_{0}^{w} \mu(\xi) a(\xi) d \xi\right).$$   

Taking $P(w) + i Q(w) = \int_{0}^{w} \mu(\xi) a(\xi) d \xi$ and $c = \alpha + i \beta$ we obtain 
$$x(u,v) = 2 [(1 + \alpha) P(w) - \beta Q(w)] \; \; \mbox{ and } \; \; y(w) = 2[(1-\alpha) P(w) + \beta Q(w)].$$       
\end{corol} 
\begin{proof} Since for $a(w) \ne 0 \ne b(w)$, with  $\frac{a(w)}{b(w)}$ entire  and such that 
$\Im(\frac{a(w)}{b(w)}) \neq 0$ (Lemma \ref{200}), the map $\frac{a(w)}{b(w)}$ does not cover whole the complex plane, then from Little Picard Theorem, it follows that
$\frac{a(w)}{b(w)}$ is constant. Under the hypotheses that constant can not be 0, 1 neither -1. 
\end{proof}

\begin{rema} We observe that Corollary \ref{300} has a weakness because while in Theorem \ref{32} the equation (\ref{30}) gives us the inversion function which is linear, and which 
we can use to construct the graphic over whole the complex plane $\complex$,
 Corollary 4.7 can not guarantee that we have a 
graphic over all complex plane, since it could exist ramifications. For instance, taking $a(w) = e^{w}$ and $\mu =1$, 
we obtain 
$P(u,v) = e^{u} \cos v$ and $Q(u,v) = e^{u} \sin v$. So, $x(u,v) = 2[ (1 + \alpha) e^{u} \cos v - \beta e^u \sin v]$ and 
$y(u,v) = 2 [(1-\alpha) e^u \cos v + \beta e^u \sin v]$, which are periodic functions in the variable $v$.
\end{rema}

In the next theorem we answer the question whether there exist a non-flat solution which is entire graphic surface of second type.
In fact,  we argue that if $a = a(w)$ is a given holomorphic function defined in whole $\complex$ and such that $a(w) \neq 0$, 
then we can take the holomorphic function $\mu(w) = \frac{1}{a(w)}$ and take also $f_{w} = \mu W(a(w),c a(w))$, with constant $c \in \complex \setminus \{0,1,-1\}$. Then next we will show that in this case, it can exist points in the surface such that the 
 Gauss curvature is not zero.   
 
 \begin{theor}\label{500}
Let $a = a(w)$ be a holomorphic function defined in whole the plane $\complex$ such that $a(w) \neq 0$ for each $w \in \complex$. For 
$c = \alpha + i \beta \in \complex \setminus \real$ we take $b(w) = c a(w)$ and $\mu(w) = \frac{1}{a(w)}$. 
Then the surfaces given by 
\begin{equation}\label{400}
f(w) = X_{0} + 2 \Re \int_{0}^{w} \left(1 + c, \frac{1}{a(\xi)} + c a(\xi), 
i \left(\frac{1}{a(\xi)} - c a(\xi)\right), 1 - c\right) d \xi,
\end{equation} 
are minimal surfaces of $\mink$, which represent  graphic  of type $X(x,y) = (x, A(x,y), B(x,y), y)$, where
the transformation of coordinates is given by  $x_{w} = (1 + c)$ and $y_{w} = (1 - c)$. 

Moreover, in this case, the Gauss curvature $K(f)(w) = 0$ if and only if $a_w(w) = 0$. 
Therefore, assuming that $a = a(w)$ is not a constant function, there exists $p \in S$ such that $K(p) \neq 0$.  
Again, there is not a hyperplane containing the surface $S$.
\end{theor}
\begin{proof} 
By integration we obtain $x = 2 \Re(((1 + \alpha) + i\beta)(u + iv)) = 2 [(1 + \alpha)u - \beta v]$ and $y = 2[(1 - \alpha)u + \beta v]$. That means,
\begin{equation}
\left[\begin{matrix} u \\ v \end{matrix} \right] = 
\frac{1}{4 \beta} \left[\begin{matrix} \beta & \beta \\ \alpha - 1 & \alpha + 1 \end{matrix} \right] \; 
\left[\begin{matrix} x \\ y \end{matrix} \right].
\end{equation} 

Since $a, b$ and $\mu$ are holomorphic functions, formula (\ref{400}) in the $(x,y)$-coordinates, represents 
a graphic minimal surface of  second type. Moreover, since the Gauss curvature is given by 
$K(f) = - \frac{1}{\lambda^2}\Delta \ln  \lambda$ where 
 $\lambda^{2} = 4 \vert \frac{1}{a \overline a} - c \vert^{2}$, it follows that $\Delta \ln \lambda \neq 0$ 
in points where $a_w(w) \neq 0$. Hence in those points $K(f) \neq 0$. 

Next, by integration we can obtain the components functions $A(w) = f^{1}(w)$ and 
$B(w) = f^{2}(w)$, and through of the transformation  of  coordinate we get the explicit representation as graphic surface. 

 Finally we note that it is needed to assume $c \not\in \real$, since we can not have  
$x_{w} y_{\overline{w}} - x_{\overline{w}} y_{w} = 0$. It  is also impossible to obtain a timelike vector $v \in \mink$ such 
that $\lpr{v}{f_{w}} = 0$, so we have the real spacial property of the surface in $\mathbb R^4_1$.
\end{proof}

So from Theorem \ref{500} one can construct a classe of minimal graphic surfaces  of second type, whose Gauss curvature is not null in any 
point of the surface. That means the property of Bernstein does not hold in this case. The following explicit example illustrates this fact.

\begin{example} We use Theorem \ref{500} to construct  second type of minimal graphic surfaces. 
Let $a = e^{w}$ and  \ $b = e^{i\theta} a$ \ for $\theta \in (0,\pi)$. Then the expression of $W(a,b)$ is
$$W(a,b) = ((1 + e^{i\theta})e^{w}, 1 + e^{i\theta} e^{2w},i(1 - e^{i\theta} e^{2w}), (1 - e^{i\theta})e^{w}).$$ 
Taking the factor $\mu(w) = e^{-w}$,  the integral representation (\ref{400}) is given by
$$f(w) = 2 \Re \int_{0}^{w} \left(1 + e^{i\theta}, e^{-\xi} + e^{i\theta} e^{\xi}, 
i(e^{-\xi} - e^{i\theta} e^{\xi}), 1 - e^{i\theta}\right) d \xi,$$
or more explicitly
\begin{equation}\label{600}
f(u,v) = 2 ((1+\cos \theta) u - v \sin \theta, -e^{-u} \cos v + e^u (\cos v \cos \theta - \sin v \sin \theta),
\end{equation}
$$ - e^{-u} \sin v + e^u(\sin v \cos \theta + 
\cos v \sin \theta), (1- \cos \theta) u + v \sin \theta)).
$$ 
Now making the transformation of  coordinates  $x_w = 1+ e^{i\theta}$ and $y_w = 1- e^{i\theta}$, we get $$x = 2[(1+\cos \theta) u - 
v \sin \theta] \ \ \ \  {\rm and } \ \ \ \  \ y = 2[(1- \cos \theta) u + v \sin \theta],$$
and hence the graphic minimal surface is given by
$
X(x,y) = (x, A(x,y), B(x,y), y)$
where the functions $A(x, y)$ and $B(x,y)$ are given by the second and third component of formula (\ref{600}) with
$$
u = \frac{x+y}{4}   \ \ \ \ \ \  { \rm and } \ \ \ \  v = \frac{1}{4 \sin \theta}[(-1 + \cos \theta)x + (1 + \cos \theta)y].
$$ 
Finally we observe that since $a_w = e^w$ never vanishes, for all the points of the surface one gets that $K(p) \ne 0$.
\end{example}

\section{The construction of the conjugated surface $(M,Y)$}

We dedicate this section for looking the explicit expression of the conjugated surface to a minimal spacelike surface   $(M,X)$ of $\mink$, using the Weierstrass notation.
We start defining a special operator on tangent bundle $TS$ to a surface, as follows. 

\begin{dfn}
Let $(M,X)$ be a spacelike surface with line element $ds^{2}(X) = E dx^{2} +2 F dx dy + G dy^{2}$, and $TS$ be its tangent bundle, 
where, pointwise, $\{X_{x}(p),X_{y}(p)\}$ is a base of $T_{p}S$.  Let $J : TS \longrightarrow TS$ be the function given by 
\begin{equation}\label{obrigada}
J(V) = \frac{1}{\sqrt{EG - F^{2}}}\left(\lpr{X_{x}}{V} X_{y} - \lpr{X_{y}}{V} X_{x}\right).
\end{equation}  
\end{dfn}

\begin{prop}
Let $J : TS \longrightarrow TS$ be the function given by the equation (\ref{obrigada}). Then $\forall V \in TS,$  the following equations are satisfied: 
$$\lpr{V}{J(V)} = 0, \; \;  \; \; \lpr{J(V)}{J(V)} = \lpr{V}{V} \; \; \mathrm{ and } \; \; J(J(V)) = - V.$$ 
\end{prop}
\begin{proof}
The first equation follows from $\sqrt{EG - F^{2}} \; \lpr{V}{J(V)} = \lpr{X_{x}}{V} \lpr{X_{y}}{V} - \lpr{X_{y}}{V} \lpr{X_{x}}{V} = 0$. For getting second equation
we take the values of $J$ in the basis, namely, 
\begin{equation}\label{912}
J(X_{x}) = \frac{1}{\sqrt{EG - F^{2}}}\left(E X_{y} - F X_{x}\right) \; \mbox{ and } \; 
J(X_{y}) = \frac{1}{\sqrt{EG - F^{2}}}\left(F X_{y} - G X_{x}\right).
\end{equation} 
Then 
$$\lpr{J(X_{x})}{J(X_{x})} = E, \ \ \ \ \ \  \lpr{J(X_{y})}{J(X_{y})} = G \ \ \ \  {\rm and} \ \ \lpr{J(X_{x})}{J(X_{y})} = F.$$
Now we note that  from the pointwise bi-linearity of $<,>$, it follows the pointwise linearity of $J$.  Therefore if $V= a X_x + bX_y$
the second 
equation of the proposition holds.

\vspace{0.1cm}

The third equation follows directly from the linearity and from the facts
$J(J(X_x)) = - X_x$  and $J(J(X_y)) = - X_y.$
\end{proof} 
\vspace{0.1cm}

We observe that if $S = (M,X)$ be a spacelike surface of $\mink$, the vector $1$-form associated to $S$ 
is given by $\beta = X_{x} dx + X_{y} dy$. 
Therefore, by definition $J(\beta)$ is the $1$-form given by 
\begin{equation}\label{911}
J(\beta) = J(X_{x}) dx + J(X_{y}) dy.
\end{equation} 

\vspace{0.3cm} Next we related the operator $J$ with the special normal frame $\{\tau, \nu\}$ in $\mathbb R^4_1$.

Let $l = \mathfrak{X}(v_{1}, v_{2}, v_{3})$ be the exterior product in 
$\mink$ of a set of vectors $\{v_1, v_2, v_3\}$. By definition, since $\Omega(\mink) = (-dx^{0}) \wedge dx^{1} \wedge dx^{2} \wedge dx^{3}$ 
is the volume form, then $l$
is defined by 
$$\lpr{l}{w} = \Omega(\mink)(v_{1}, v_{2}, v_{3}, w), \ \ \ \  \forall w \in \mink.$$
Then the $J$ operator is equivalent to $J(V) = \mathfrak{X}(\tau,\nu,V)$.

\begin{theor}\label{211}
Let $S = (M,X)$ be a spacelike surface of $\mink$ and let $\beta = X_{x} dx + X_{y} dy$ be the vector $1$-form associated to $S$. Then 
\begin{equation}\label{900}
J(\beta) = \frac{-F dx - G dy}{\sqrt{EG - F^{2}}} \; X_{x} + \frac{E dx + F dy}{\sqrt{EG - F^{2}}} \; X_{y}.
\end{equation} 
The $1$-form $J(\beta)$ is closed if and only if $(M,X)$ is a minimal spacelike surface.
\end{theor}
\begin{proof}
The equation (\ref{900}) follows from equations (\ref{912}) and (\ref{911}).  

For the second statement, we use the representation of the operator $J$ as an exterior product, to obtain $$J(\beta) = 
\mathfrak{X}(\tau,\nu, X_x dx + X_y dy) = \tau \times \nu \times \beta.$$ 

Now, since $d\beta = 0$, we get the exterior derivative 
$d J(\beta) = ((d \tau) \times \nu \times \beta) + (\tau \times (d \nu) \times \beta)$. 

Next we will calculate explicitly $dJ(\beta)$. For that, we use $d \tau = \tau_{x} dx + \tau_{y} dy$,
 $d \nu = \nu_{x} dx + \nu_{y} dy$, the Weingarten formulas (3), (4) and the anti-commutative properties of the 
exterior product in $\mink$ and of the exterior product of $1$-forms, to obtain  
$$
d(J \beta) = (h^1_1 + h^2_2) (X_x \times \nu \times X_y) dx \wedge dy + (n^1_1 + n^2_2) 
(\tau \times X_x \times X_y) dx \wedge dy.
$$
Since $X_{x} \times \nu \times X_{y} = - \sqrt{EG - F^{2} \;}  \tau$ and $\tau \times X_{x} \times X_{y} = \sqrt{EG - F^{2} \;} \nu$ 
one gets
\begin{equation}\label{100}
d J(\beta) =  - 2 H_{X} \sqrt{EG - F^{2} \;} dx \wedge dy.
\end{equation} 
Hence it follows from equation (\ref{100}) that, $d J(\beta) = 0$ if and only if $(M,X)$ is minimal. 
\end{proof}

Theorem \ref{211} allows us to establish the next corollary which shows the explicit expression of the minimal
conjugate spacelike surface $(M,Y)$ in $\mink$.  It comes from the fact that since $J(\beta)$ is a closed 1-form in a connected 
simply-connected open subset  of $\mathbb C$ then it is  exact. 
\begin{corol}
Let $M$ be a connected and simply connected open subset of the plane $\complex$, and let $(M,X)$ be a solution of the 
minimal graphic 
equations (\ref{10}). The integral representation (\ref{3}) can be extended to $Z = X + i Y \in \complex^{4}$ by 
\begin{equation}
Z(x,y) = Z(x_{0},y_{0}) + \int_{z_{0}}^{z} \beta + i J(\beta), 
\end{equation}
where 
\begin{equation}\label{7}
Y(x,y) = Y(x_{0},y_{0}) + \int_{z_{0}}^{z} \frac{-F dx - G dy}{\sqrt{EG - F^{2}}} \; X_{x} + 
\frac{E dx + F dy}{\sqrt{EG - F^{2}}} \; X_{y}.
\end{equation} 
Moreover, $Y$ gives us the parametrization of the conjugated minimal spacelike surface $(M,Y)$ of  $\mink$. 
\end{corol}

\begin{proof} Since $J(dY) = J(J(dX) = - dX = - (X_xdx + X_y dy)$ is a closed vector 1-form, from Theorem \ref{211} it follows that $H_Y(p) =0$ for each $p \in M$.
\end{proof} 

\vspace{0.1cm}
\begin{example}
Let $X(x,y) = (0,x \cos y, x \sin y, y)$ be a parametrizated Helicoid of $\euclidean$. The conjugated  minimal spacelike 
surface, given by 
equation (\ref{7}) 
with $Y(0,0) = (0, 0, 1, 0)$, is the Catenoid given in coordinates by
$$Y(x,y) = (0, -  \sqrt{1 + x^{2}}  \sin y, \   \sqrt{1 + x^{2}}  \cos y, \ \ln(x + \sqrt{1 + x^{2}})).$$ 
In fact, from $X_{x} = (0,\cos y, \sin y, 0)$ and $X_{y} = (0, -x \sin y, x \cos y, 1)$ it follows that $E = 1$, $F = 0$ and 
$G = 1 + x^{2}$. Now from the integral equation (\ref{7}) we obtain 
$$d Y = \frac{1}{\sqrt{1 + x^{2}}}(0,-x \sin y, x \cos y, 1) dx - \sqrt{1 + x^{2}}(0, \cos y, \sin y, 0) dy.$$
Hence  by integrating $Y_x =  \frac{1}{\sqrt{1 + x^{2}}}(0,-x \sin y, x \cos y, 1)$  and   $Y_y =  - \sqrt{1 + x^{2}}(0, \cos y,  \sin y, 0)$,
we get  the Catenoid surface $(\real^{2}, Y(x,y))$. 

\vspace{0.1cm}
Moreover, if $x \geq 0$ we have the part corresponding to $Y^{3} \geq 0$ and, if $x \leq 0$ we have the part corresponding to 
$Y^{3} \leq 0$. Both surfaces $(\real^{2},X(x,y))$ and its conjugated $(\real^{2}, Y(x,y))$ are ramified. 

\vspace{0.1cm}
Finally, if we make $x = \sinh u$ and $y = v$, we obtain 
$$\tilde{X}(u,v) = (0, \sinh u \cos v, \sinh u \sin v, v) \; \  \mbox{ and } \;  \ 
\tilde{Y}(u,v) = (0, -  \cosh u \sin v,  \cosh u \cos v, u),$$ 
in the isothermic coordinates $(u,v)$.  As it is expected it follows $\tilde X_u = - \tilde Y_v$ and  $\tilde X_v =  \tilde Y_u$.
\end{example} 

\vspace{0.1cm}
Next example shows an applicability of the $J$ operator.

\begin{example} Let $X(x,y) = (x \cosh y, x \sinh y, f(x), 0)$ be a graphic type of hyperbolic rotation in $\mathbb R^3_1$ in hyperbolic polar coordinates.  Since $X_x = (\cosh y, \sinh y, f'(x), 0)$ and $X_y =(x \sinh y, x \cosh y, 0, 0)$, we get
$$
E(x,y) = -1 + (f'(x))^2 > 0,  \ \ F(x,y) =0, \ \ G(x,y)= x^2 >0,  \  \ W = x\sqrt{(f')^2 -1}.
$$

Hence, a needed condition for obtaining a minimal spacelike surface $Y$, in terms of the operator $J$, is  $$dJ(\beta) = dJ(dX) = 0.$$
Then the $Y^2$-coordinate gives us the equation $\frac{x f'}{\sqrt{-1 + (f')^2}} = k$ or more specifically
$$
(k^2 -x^2)(f')^2 = k^2 \ \ \  {\rm with} \ k \in \mathbb R - \{0\}, \  \ |x| < |k|.
$$
Now integrating, one obtains 
$$
f(x) = b + (\pm k) arcsin(x/k).
$$
Assuming  $k>0$ and $b=0$, we get the parametric surface 
$$X(x,y) = (x \cosh y, x \sinh y, k. arcsin(x/k), 0).$$
If we take $x = k \sin u$ and $y= v$, we get the correspondent minimal parametric surface with isothermic parameters given by
$$
f(u,v) = k(\sin u \cosh v, \sin u \sinh v, u, 0), 
$$
where $E(u,v) = k^2 \sin^2 u = G(u,v)$ and with lightlike singularities for $f_u(u,v)$ when $u = n \pi$, $n \in \mathbb Z$.

In similar way we get the surface given by
$$
g(u,v) = k (\cos u \cosh v, \cos u \sinh v, v, 0),
$$
which is a minimal ruled surface with the same metric tensor, it is a type of hyperbolic  helicoid surface of $\mathbb R^4_1$.
\end{example}

\vspace{0.1cm}
\subsection{Generalized Cauchy-Riemann equations over $(M,X)$} In this subsection we continue studying the local geometry of the spacelike surfaces in $\mathbb R^4_1$. In particular in this first part, 
we identify the generalized Cauchy-Riemann type equations over the surface 
$(M,X)$ when the parameters are not isothermic, and then we obtain the needed conditions to extend in continua way any local solution of those equations. 

\vspace{0.2cm}
For starting, we observe that if we have a sub-surface $f(U) \subset X(M)$ with 
isothermic parameters $w = (u,v) \in U$ such that $X(x,y) = f(u(x,y),v(x,y))$, then 
$$\frac{\partial X}{\partial x} = \frac{\partial u}{\partial x} \; f_{u} + \frac{\partial v}{\partial x} \; f_{v}
\ \ \ \; \; \mathrm{ and } \; \; \ \ \ 
\frac{\partial X}{\partial y} = \frac{\partial u}{\partial y} \; f_{u} + \frac{\partial v}{\partial y} \; f_{v}.$$

\vspace{0.1cm}

\begin{lemma}
For each local solution of the equations 
\begin{equation}\label{222}
\frac{\partial w}{\partial y} = \alpha(x,y) \; \frac{\partial w}{\partial x} \; \mbox{ where } \; 
\alpha(x,y) = \frac{F(x,y) + i \sqrt{E(x,y) G(x,y) - F^{2}(x,y)}}{E(x,y)},  
\end{equation} 
 in a neighborhood $U \subset M$ of a point $p \in M$, there exists a parametric isothermic sub-surface $(U,f)$ of $(M,X)$ such that 
 $X(x,y) = f(u(x,y),v(x,y))$. 
Moreover, $\alpha \overline{\alpha} = \frac{G}{E}$. 
\end{lemma}
\begin{proof}
Let $W = \sqrt{EG - F^{2}}$ be the area function in coordinates $z = x + i y \in U$. 
Taking the operator $J$, since 
 $J(f_{u}) = f_{v}, \ J(f_v) = - f_u$, it follows 
$$J(X_{x}) =  u_x J(f_u) + v_x J(f_v)  = u_x f_v - v_x f_u.
$$
Hence by equation (\ref{912}) one gets 
$$u_x f_v - v_x f_u = \frac{E}{W}(u_{y} f_{u} + v_{y} f_{v}) - \frac{F}{W}(u_{x} f_{u} + v_{x} f_{v}).$$

 From this last equation, we obtain the following equations with matrix 
representation 
\begin{equation}\label{36}
\left[\begin{matrix} u_{y} \\ v_{y} \end{matrix}\right] = \left[\begin{matrix} F/E & -W/E \\ W/E & F/E \end{matrix}\right] \; 
\left[\begin{matrix} u_{x} \\ v_{x} \end{matrix}\right] \; \; \mbox{ and } \; \; 
\left[\begin{matrix} u_{x} \\ v_{x} \end{matrix}\right] = \frac{E}{G} \left[\begin{matrix} F/E & W/E \\ -W/E & F/E \end{matrix}\right] \; 
\left[\begin{matrix} u_{y} \\ v_{y} \end{matrix}\right].  
\end{equation}
Now we observe that the square matrices of order $2\times 2$ of these equations are the matrix representation of a complex number. Therefore 
we can write $$u_{y} + i v_{y} = \frac{F + i W}{E} \; (u_{x} + i v_{x}),$$ 
that is equation (\ref{222}).    
\end{proof}

We note that if the $(x,y)$-coordinates are already isothermic coordinates  then $\alpha = i$ and so equations (\ref{222}) for $(M, X)$ becomes to
the classic expression of the Cauchy Riemann equations, namely,  $u_y = - v_x$ and $u_x =  v_y$.
So we will  call equations (\ref{36}) or (\ref{222}) as the generalized Cauchy-Riemann equations for $(M, X)$. 

Then as expected we have the following definition-corollary.

\begin{corol}
A smooth function $h = \varphi + i \psi : S \longrightarrow \complex$ is generalized holomorphic over the Riemann surface $S = X(M)$ if and only if  
\begin{equation}\label{certa}
\left[\begin{matrix} \varphi_{y} \\ \psi_{y} \end{matrix}\right] = \left[\begin{matrix} F/E & -W/E \\ W/E & F/E \end{matrix}\right] \; 
\left[\begin{matrix} \varphi_{x} \\ \psi_{x} \end{matrix}\right] \; \; \; {\rm or } \; \; \;
\frac{\partial h}{\partial y} = \alpha \frac{\partial h}{\partial x},
\end{equation}
where $W = \sqrt{EG - F^{2}}$.
\end{corol}  
\begin{proof}
Since in an isothermic neighborhood $(U,\tilde h)$ the function $\tilde{h}(u,v)$ is holomorphic in the sense of complex variable if and only if 
$h(x,y)$ is holomorphic over $S$, we have
$$\frac{\partial h}{\partial x} = \frac{\partial \tilde{h}}{\partial u} (u_{x} + i v_{x}) \; \; \mbox{ and } \; \; 
\frac{\partial h}{\partial y} = \frac{\partial \tilde{h}}{\partial u} (u_{y} + i v_{y}),$$ 
because $i \tilde{h}_{u} = \tilde{h}_{v}$ holds for $\complex$-holomorphic functions. Therefore $h_{y} = \alpha h_{x}$ 
follows from the definition of the function $\alpha(x,y)$.   The conversely is immediate. 
\end{proof}

We note that for smooth function $h = \varphi + i \psi : U \subset S \longrightarrow \complex$ is a  generalized holomorphic if and only if in isothermic
coordinates $(u,v)$ the harmonic functions $\varphi, \psi$ are conjugated harmonic functions satisfying the usual Cauchy-Riemann
equations. 

\vspace{0.1cm}
If we use the operator $J$, we can also give an equivalent definition, namely:  $h$ is a   generalized holomorphic function if and only if 
$$
dJ(\varphi(x,y)) = d \psi(x,y) \ \ \ \ \ \ \ \  {\rm and} \ \ \ \ \ \ dJ(\psi(x,y)) =- d \varphi(x,y).
$$
They are generalized harmonic functions conjugated each other.

\vspace{0.3cm}

Next we are interested in relating the isothermic neighborhood $(U, f)$ with the Weierstrass datas $a(w)$ and $b(w)$ for graphic spacelike surfaces in $\mathbb R^4_1$.

In fact, fixing the semi-rigid referential associated to $(M,X)$  given by 
$$\M_{0} = \{l_{0}(b(p)), e_{1}(p),e_{2}(p),l_{3}(a(p))\}$$ where 
$$e_{1}(p) = \frac{1}{\sqrt{E}} \frac{\partial X}{\partial x} \; \; \mbox{ and } \; \; 
e_{2} = J(e_{1}) = \frac{1}{\sqrt{E}} J(X_{x}),$$
we obtain the next result.

\begin{prop}\label{32}
Let $S = (M,X)$ be a solution of the minimal graphic equation in $\mink$ and $(U,f)$ be a given locally isothermic 
sub-surface of $S$. Let $r(u,v)$ be a real-valued function and  $\M(\vartheta) = \{l_{0}(b), e_{1},e_{2},l_{3}(a)\}_{(u,v)}$  be 
the semi-rigid referential associated to $f_w(w) = \mu(w) W(a(w),b(w)) = r(w)(\hat{e}_{1}(w) - i \hat{e}_{2}(w))$. Then the following relation is hold:
$$\hat{e}_{1}(w) - i \hat{e}_{2}(w) = (\cos \vartheta \; e_{1} + \sin \vartheta \; e_{2}) - 
i(-\sin \vartheta \; e_{1} + \cos \vartheta \; e_{2})) = e^{i\vartheta}(e_{1} - i e_{2}).$$ 
\end{prop}

\vspace{0.3cm}

From Proposition \ref{32} it follows that if the coordinates $(M,X)$, $(U,f)$ and $(\tilde{U},\tilde{f})$ around a point 
$p \in f(U) \cap \tilde{f}(\tilde{U})$ are related by the equations  
$$X(x,y) = f(u(x,y),v(x,y)) = f \circ \Phi(x,y), \; \; 
X(x,y) = \tilde{f}(\tilde{u}(x,y),\tilde{v}(x,y)) = \tilde{f} \circ \tilde{\Phi}(x,y),$$ 
then the transition functions are given by 
\begin{equation}
f \circ \Phi(x,y) = \tilde{f} \circ \tilde{\Phi} \; \; \mbox{ therefore } \; \; 
\Psi = \tilde{\Phi} \circ \Phi^{-1} = \tilde{f}^{-1} \circ f.
\end{equation} 

Now, applying the Proposition  \ref{32}, we obtain: 
$$\frac{1}{\hat{r}} f_{w} = e^{i \hat{\phi}}(\hat{e}_{1} - i \hat{e}_{2}) \; \; \mbox{ with } \; \; 
\frac{1}{\tilde{r}} \tilde{f}_{\tilde{w}} = e^{i \tilde{\phi}}(\tilde{e}_{1} - i \tilde{e}_{2}),$$ 
which imply that the angle functions are related each other by the equation:
\begin{equation} \label{34}
\hat{\phi}(u,v) - \tilde{\phi} \circ \Psi(u,v) = \hat{\vartheta}(u,v) - \tilde{\vartheta} \circ \Psi(u,v). 
\end{equation} 

\vspace{0.3cm}

Now we have the following facts, which come from equation  (\ref{34}).
\vspace{0.1cm}

(1) {\em If two holomorphic functions agree with each other along a Jordan arc, then they agree with each other along all connected 
component of this arc}. 

From (1) we obtain.

(2) {\em If $(U,f)$ and $(\tilde{U},\tilde{f})$ agree with each other along an Jordan arc in $S$, they agree with each other 
along the open subset $f(U) \cap \tilde{f}(\tilde{U})$}.

(3) {\em The overlapping or transition map between two isothermic coordinates system for a spacelike surface of $\mink$ 
are holomorphic function in sense of complex analysis}. 

(4) {\em Each holomorphic function $h : U \subset \complex \longrightarrow V \subset \complex$ can be seen as a pointwise $\complex$-linear 
transformation $dh_{z_{0}} : T_{z_{0}} \complex \longrightarrow T_{h(z_{0})} \complex$ that preserves oriented angles}.

\begin{lemma}\label{35}
The angle function $\tilde{\vartheta} - \vartheta$ determines the transition map of $(U,f)$ and $(\tilde{U},\tilde{f})$ 
for two isothermic parametrizations of the neighborhood $f(U) \cap \tilde{f}(\tilde{U}) \subset S$, around $p \in S$.
\end{lemma}

From Lemma \ref{35} it follows the next result about the extension of local solutions of equation (\ref{222}).
   
\begin{prop}\label{1020}
Let $w, \tilde w$ two local solutions of equation (\ref{222}), around a point $p \in S$,  with 
$w_{y} = \alpha w_{x}$ and $\tilde{w}_{y} = 
\alpha \tilde{w}_{x}$. If  $w_{x} = \tilde{w}_{x}$ then $w_{y} = \tilde{w}_{y}$. 

\vspace{0.1cm}
Therefore, all local solution of the equation (\ref{222}) can be continuously extended whenever $E(x,y) > 0$ and $\sqrt{EG -
 F^{2}}(x,y) > 0$. 
\end{prop} 
\begin{proof} The conclusions are  immediate. 
\end{proof}

\vspace{0.2cm}
Next  we prove that the solutions $w = (u,v)$ of the generalized Cauchy-Riemann equations (\ref{36}) or (\ref{222}),  are of the form of Nitsche type functions (equation (8), page 23 of \cite{7}).

\begin{theor}\label{1010}
The solution for equations (\ref{36}) are given by Nitsche type functions, that means
\begin{equation}
\begin{matrix}\label{390}
u = u(x,y) = x + \int_{z_{0}}^{z} \frac{E dx + F dy}{W} \\ v = v(x,y) = y + \int_{z_{0}}^{z} \frac{F dx + G dy}{W}. 
\end{matrix}
\end{equation}
Moreover, from equations (\ref{390}), it is possible to obtain global isothermic coordinates $(U,f)$ for the surface $S = X(M)$.
\end{theor} 
\begin{proof}
In fact, since 
$$\frac{\partial u}{\partial x} = \frac{W + E}{W}, \; \; \; \; \frac{\partial u}{\partial y} = \frac{F}{W}, \; \; \; \;  
\frac{\partial v}{\partial x} = \frac{F}{W}, \; \; \; \; \frac{\partial v}{\partial y} = \frac{W + G}{W}$$ 
the matrix equation (\ref{36}) is satisfied. In fact, remembering that $W^{2} + F^{2} = EG$, we obtain
$$\left[\begin{matrix} F/W \\ (W + G)/W  \end{matrix} \right] = \left[\begin{matrix} F/E & -W/E \\ W/E & F/E \end{matrix}\right] \; 
\left[\begin{matrix} (E + W)/W \\ F/W  \end{matrix} \right].$$
Now, since the solution for equations (\ref{36}) are in the form (\ref{390}), we obtain that the local isothermic parameters $(u,v)$ can be extended globally for the surface $S$ since the conditions of Proposition \ref{1020} are hold.
\end{proof}

We highlight in this moment that our generalized Cauchy-Riemann equations (\ref{36}) and its solutions in (\ref{390}) 
can be applicated when we want to
construct the  conjugate minimal spacelike surface $(M,Y)$ (\ref{7}), since the solutions 
 (\ref{390}) involve terms of the local parametrization of $(M,Y)$. 
 
\vspace{0.3cm}
Finally we have the following corollary for equations of minimal graphic surfaces in $\mathbb R^4_1$.
\begin{corol}\label{9080}
If $S = (\real^{2},X)$ is a solution of the minimal graphic equation (\ref{10}) then, for all $p \in S$, 
the functions $a(w)$ and $b(w)$ satisfy 
either  $b(p) = c a(p)$ with  $c \notin \{-1,1\}$ or  $a(p) b(p) = c$ with  $\Im(c) \neq 0$ for some constant   $c \in \complex.$

The Bernstein Theorem and the Calabi Theorem follows from that $c \neq 1$ and $c \neq -1$ and for the second type of surfaces from 
$\Im(c) \neq 0$. 

\vspace{0.2cm}

Finally, if as a submanifold of the topological vector space $\real^{4}$ there exists $S = (\real^{2}, X)$ such that with the induced metric 
of $\mink$, is a spacelike graphic solution in connected and simply connected open subset $M \subset \complex$, with the condition that
in some point $p \in S$ the following statement fails:   

\lq \lq {\em  either $b(p) = c a(p)$  with  $c \notin \{-1,1\}$ \ or \ $a(p) b(p) = c$ \ with \ $\Im(c) \neq 0$  and  for some constant 
$c \in \complex$}", 

then the points $X(x,y)$ where $E G - F^{2} = 0$, are points such that the tangent planes of $X(\real^{2})$ are tangent to the lightcone 
of $\mink$.
\end{corol}

So from the second part of Corollary  \ref{9080}, we have found conditions to create graphic minimal spacelike surfaces 
which have new type of singularities, it called lightlike  singularities, as 
defined by Kobayashi in \cite{4}.  Those singularities are points where the tangent plane of the surface is also tangent
to the lightcone of $\mink$.  

\section{A Particular Family of Minimal Surfaces of $\mink$} 
In this  section we construct examples of minimal spacelike surfaces in $\mink$ which are very close related 
to surfaces in $\euclidean$ and $\submink$. 

\vspace{0.1cm}
For the representation $f_{w} = \mu(a + b, 1 + a b, i(1 - a b),a - b)$ with $\mu, a, b$  holomorphic functions from $M$ 
into $\complex$, with $M$ being connected and simply connected open subset of the complex plane, we assume the relation 
\ $b = a e^{i\theta}$   for a parameter $\theta \in \real$. 
\begin{dfn}
A $\theta$-family is a set of minimal surfaces defined on a connected and simply connected domain $M \subset \complex$, linking each other by a parameter $\theta \in \real$, given by the following equation 
\begin{equation}
F(\theta;w) = P_{0} + 2 \Re \int_{w_{0}}^{w} \mu(\xi)\left((1 + e^{i\theta})a(\xi), 1 + e^{i\theta} a^{2}(\xi), i(1 - e^{i\theta} a^{2}(\xi)), 
(1 - e^{i\theta})a(\xi)\right) d \xi.
\end{equation}

\vspace{0.1cm}
When $\theta = 0$ we say that the surface of  $\submink$,  given by $X(w) = F(0;w)$, is the initial surface of the family, and 
when $\theta = \pi$ we say that the surface of $\euclidean$, given by $Y(w) = F(\pi;w)$, is the associated surface of the initial 
surface of the family.  
\end{dfn}

\begin{lemma}
For a $\theta$-family $(M,F(\theta;w))$ of minimal spacelike isothermic parametric surfaces in $\mink$ the equations that 
related the initial surface $(M,X)$ and the associated surface $(M,Y)$,  are given by:
\begin{equation}\label{141}
\frac{\partial Y^{3}}{\partial w} = \frac{\partial X^{0}}{\partial w}, \; \; \; \; 
\frac{\partial Y^{1}}{\partial w} = -i \frac{\partial X^{2}}{\partial w} \; \; \;  {\rm and} \ \ \ \
\frac{\partial Y^{2}}{\partial w} = i \frac{\partial X^{1}}{\partial w}.
\end{equation}
\end{lemma}
\begin{proof}
The equations of lemma follows from $X_{w} = \mu(2a, 1 + a^{2}, i(1 - a^{2}), 0)$ and $Y_{w} = \mu (0, 1 - a^{2},i(1 + a^{2}), 2a)$.   
\end{proof}

Now we construct an example for these equations: 

\begin{example}\label{12}
Let $(M,X)$ be the minimal  spacelike surface of $\submink$ given, in isothermic parameters, by 
$$X(u,v) = (u, \sinh u \cos v, \sinh u \sin v,0).$$ 

Since $X_{u} = (1, \cosh u \cos v, \cosh u \sin v,0)$ and $X_{v} = (0,-\sinh u \sin v, \sinh u \cos v,0)$ we obtain  
$\lambda^{2}(X) = \sinh^{2} u$. We assume that $(u,v) \in M$ for $u > 0$. 

Therefore, it follows $X_{w} = \frac{1}{2}(1, \cosh w, -i \sinh w,0).$ To obtain the associated surface we find the functions $a(w)$ and $\mu(w)$. In fact, since 
$$2 \mu a = \frac{1}{2}, \; \; \; \; \mu (1 + a^{2}) = \frac{\cosh w}{2}, \; \; \; \; i \mu (1 - a^{2}) = -i \frac{\sinh w}{2},$$  
it follows that $4 \mu(w) = e^{- w}$ and $a(w) = e^{w}$. 

For obtaining the associated surface $(M,Y)$, we use $Y_{w} = \mu (0, 1 - a^{2},i(1 + a^{2}), 2a)$, and so the surface is such 
$Y_{w} = \frac{1}{2}(0, - \sinh w, i \cosh w, 1)$. Hence the holomorphic integral curve is given by 
$$\tilde{Y}(w) = \frac{1}{2} (0, -\cosh w, i\sinh w ,w).$$ 

Thus, the real part of $\tilde Y$ gives us a Catenoid of $\euclidean$ parametrized by 
$$Y(u,v) = (0, -\cosh u \cos v, - \cosh u \sin v, u) \; \; \mbox{ com } \; \; \lambda^{2}(Y) = \cosh^{2} u.$$ 

Now, we look for the representation of those two associated surfaces as graphics of first type. In fact, 
for $(M,X)$ and the representation $P(x,y) = (A(x,y), x, y, 0)$:  It takes  
$$x = \sinh u \cos v \ \ \ \ \  and \ \ \ \  y = \sinh u \sin v.$$ Therefore $\sinh u = \sqrt{x^{2} + y^{2}}$. 
For $(M,X)$, we obtain that function $A$ in the graphic representation is given by  
$$A(x,y) = \ln(\sqrt{x^{2} + y^{2}} + \sqrt{x^{2} + y^{2} + 1}).$$

For $(M,Y)$ and the representation 
$Q(p,q) = (0, p, q, B(p,q))$:  It takes $$p = -\cosh u \cos v \ \ \ \ \  and \ \ \ \  q = -\cosh u \sin v.$$ So, $\cosh u = \sqrt{p^{2} + q^{2}}$. 
For $(M,Y)$ we obtain the function $B$ as given by 
$$B(p,q) = \ln(\sqrt{p^{2} + q^{2}} + \sqrt{p^{2} + q^{2} - 1}).$$ 
\end{example}

\vspace{0.3cm}

Example \ref{12}, and equations linking the initial surface $(M,X)$ and its associated surface $(M,Y)$ in the $\theta$-family, 
suggest the following result. 

\begin{lemma}\label{151}
For the associated surfaces of the $\theta$-family  given by
$$X(w) = (A(x(w),y(w)), x(w), y(w), 0) \; \; \mbox{ and } \; \; Y(w) = (0, p(w), q(w), B(p(w),q(w)),$$ the Jacobian functions 
of the transformation of coordinates, are related by 
$$\frac{\partial (x,y)}{\partial (u,v)} = \frac{\partial (p,q)}{\partial (u,v)}.$$   
\end{lemma}
\begin{proof}
From equations of associated surfaces (\ref{141}) it follows that $p_{w} = -i y_{w}$ and $q_{w} = i x_{w}$.  Then  
$p_{w} q_{\overline{w}} - p_{\overline{w}} q_{w} =  x_{w} y_{\overline{w}} - x_{\overline{w}} y_{w},$
which implies the relation $\frac{i}{2} [p_u q_v - p_v q_u] = \frac{i}{2} [y_v x_u - y_u x_v].$ 
\end{proof}

Finally, from Lemma \ref{151} and from our version of the Nitsche equations for transformation of coordinates (\ref{390}), we obtain the following result. 

\begin{theor}
The $\theta$-family transports minimal first type  graphic solutions $P(x,y) = (A(x,y), x, y, 0)$ to minimal associated
graphic solutions $Q(p,q) = (0, p, q, B(p,q))$ preserving the domain $dom (A) = dom (B) = M$. 

If $M = \complex$ then $P(\complex)$ and $Q(\complex)$ are spacelike planes of $\mink$. 

We can say that \lq \lq the Bernstein theorem holds if and only if the Calabi theorem holds". 
\end{theor}

{\bf Acknowledgments} \  The first author's research was supported by Projeto Tem\'atico Fapesp n. 2016/23746-6. S\~ao Paulo. Brazil. This paper is part
of the Ph.D. thesis of R.S. Santos \cite{6}, which was presented in Universidade de S\~ao Paulo, Brazil, in February 2021.

\end{document}